\documentclass[a4paper,12pt]{amsart}
\usepackage{amssymb}
\usepackage{latexsym}
\usepackage{amsmath}
\usepackage[english]{babel}
\usepackage{amsfonts}
\usepackage{todonotes}
\usepackage{enumerate}
\usepackage{youngtab}
\newcommand{\excise}[1]{}
\usepackage[all]{xy}
\CompileMatrices
\usepackage{graphicx,tikz}
\usepackage[active]{srcltx}
\usepackage[pdftex]{hyperref}
\newcommand{\Hom}{\operatorname{Hom}}

\newcommand{\Ind}{\operatorname{Ind}}

\newcommand{\End}{\operatorname{End}}

\newcommand{\op}{\operatorname}
\newcommand{\Gr}{\operatorname{Gr}}

\usepackage{bbm}
\newcommand\C{\mathbb C}
\newcommand\Z{\mathbb Z}

\newcommand\N{\mathbb N}
\newcommand\Q{\mathbb Q}

\newcommand\mZ{\mathbb Z}

\newcommand\la{\lambda}

\usepackage{amsthm}
\newtheorem{thm}{Theorem}[section]

\newtheorem{prop}[thm]{Proposition}
\newtheorem{cor}[thm]{Corollary}
\newtheorem{lemma}[thm]{Lemma}

\theoremstyle{definition}

\numberwithin{equation}{section}

\title{Fusion rings for quantum groups}

\author{Henning Haahr Andersen}
\author{Catharina Stroppel}

\address{HHA: Center for Quantum Geometry of Moduli Spaces, Aarhus
University, Building 530, Ny Munkegade, 8000  Aarhus C, DENMARK}
\address{CS: Department of Mathematics, University of Bonn, Endenicher Allee 60, 53115 Bonn, GERMANY}
\thanks
{HHA was supported by the Danish National Research Foundation center of Excellence, Center for Quantum Geometry of Moduli Spaces (QGM); and CS by a visiting professorship at Chicago university. We thank Troels Bak Andersen and Stephen Griffeth for comments on a preliminary version of the paper.}

\begin{document}

\begin{abstract}
We study the fusion rings of tilting modules for a quantum group at a
root of unity modulo the tensor ideal of negligible tilting modules.
We identify them in type $A$ with the combinatorial rings from
\cite{KorffStroppel} and give a similar description of the
$\mathfrak{sp}_{2n}$-fusion ring in terms of non-commutative symmetric
functions. Moreover we give a presentation of all fusion rings in
classical types as quotients of polynomial rings. Finally we also compute the fusion rings for type $G_2$.
\end{abstract}

\maketitle
\tableofcontents

\section{Introduction}
Fusion rings associated with semisimple Lie algebras have been studied from different perspectives, see e.g. \cite{BR}, \cite{BK}, \cite{Dou}, \cite{GM}, \cite{Gepner}, \cite{KorffStroppel}, \cite{W}. In this paper
we approach them via tilting modules for quantum groups $U_q(\mathfrak g)$ at complex roots of unity and relate them to previously obtained descriptions obtained in loc. cit. In particular, we give an alternative way of proving the combinatorial fusion rule for type $A$ established by C. Korff  and the second author in \cite{KorffStroppel} and  also obtain an analogous formula in the type $C$ case.

Our approach gives for all types an easy method to realize the fusion rings as quotients of polynomial rings. In particular, it provides a set of generators of the fusion ideal. In the type $A$ and $C$ cases this leads to minimal sets of generators  identical to those found in the above mentioned papers and suitable for nice combinatorics. The other cases are more involved and the naturally arising sets of generators are in general not minimal, as our case by case analysis reveals.

As a special feature of our approach we obtain two different fusion rings in each type depending on whether our quantum parameter has even or odd order. In type $G_2$ it moreover depends on whether $3$ divides this order or not, see Section \ref{G2A}; we avoid this special case for the rest of the introduction. To obtain the above mentioned fusion rings from the literature it is enough to restrict to quantum groups where the order of the root of unity is even. The ``odd" fusion rings however do not seem to have been studied previously (except implicitly in \cite{AP}). Finding a minimal set of generators for these fusion ideals seem to be even more challenging and so far no satisfactory answer is known. Based on our presentations we could verify the intriguing conjectural presentations of \cite{BK} for small ranks by computer calculations, but unfortunately are not able to prove or disprove them in general. We expect that an identification of the two presentations yields interesting new identities inside the character ring.

Let us explain our approach and results in a little more detail. We work with a simple complex Lie algebra $\mathfrak g$ with root system $R$. We choose a set of positive roots $R^+$ and then have inside the set $X$ of integral weights the cone $X^+ \subset X$ of dominant weights. Let $q \in \C$ be a root of unity of order $\ell$ and denote by $U_q(\mathfrak g)$ the quantum group  associated with $\mathfrak g$. If $\ell$ is even we set $\ell' = \ell/2$.

Inside the category of finite dimensional $U_q(\mathfrak g)$-modules there is the additive category $\mathcal T_q$ given by the {\it tilting modules}. This is a tensor category and its (split) Grothendieck ring $K$ has a $\Z$-basis consisting of the images of all isomorphism classes of indecomposable tilting modules. The latter are classified; for each $\lambda \in X^+$ there is a unique indecomposable tilting module $T_q(\lambda)$ with highest weight $\lambda$. Inside $\mathcal T_q$ we consider the tensor ideal  $\mathcal N_q$, cf. \cite{Atensor}, consisting of all {\it negligible tilting modules}, i.e. those tilting modules $Q$ for which $\op{Tr}_q(f) = 0$ for all $f \in \End_{\mathcal T_q} (Q)$, where $\op{Tr}_q$ denotes the {\it quantum trace}. The corresponding {\it fusion ring}
for $U_q(\mathfrak g)$ is  then the ring $\mathcal F_q = K/I_\ell$, where $I_{\ell}$ denotes the ideal in $K$ corresponding to $\mathcal N_q$.

To describe this fusion ring we relate it with the corresponding polynomial ring of characters. Let $W$ denote the Weyl group for $\mathfrak{g}$. The character ring  $\Z[X]^W$ has a $\Z$-basis given by the $\chi (\lambda)_{\lambda \in X^+}$, where for any $\lambda \in X$ we denote by $\chi(\lambda)$ the {\it Weyl character} at $\lambda$. These characters are linked for $w \in W$ by the formula $\chi(w \cdot \lambda) = (-1)^{\ell(w)} \chi(\lambda)$ with the {\it ``dot"-action} given by $w \cdot \lambda = w(\lambda + \rho) -\rho$ where $\rho$ is the half-sum of the positive roots. Consider the {\it fundamental alcove}
\begin{eqnarray*}
\mathcal{A}_\ell&=&\{\lambda \in X \mid 0 < \langle \lambda + \rho, \alpha^{\vee}\rangle  < \ell \text { for all } \alpha \in R^+\},
\end{eqnarray*}
if $\ell$ is odd and with $\ell$ replaced by $\ell'=\ell/2$ if $\ell$ is even. Then the tensor ideal $\mathcal N_q$ consists of those tilting modules which have no summands $T_q(\lambda)$ with $\lambda \in \mathcal A_{\ell}$. Note that this means $K/I_\ell = 0$ when $\ell, \ell'  < h$, where $h$ denotes the Coxeter number for $R$. Hence we assume always $\ell, \ell' \geq h$. In that case it is proved in \cite{AP} that $I_\ell$ may be described as
$$\left\{\sum a_\lambda \chi(\lambda) \in \Z[X]^W \mid a_\lambda = a_{s \cdot \lambda} \text { for some reflection } s \text { in a face of  } \mathcal{A}_\ell \right\}.$$
Based on this observation we use the representation theory and combinatorics of affine Weyl groups to obtain explicit descriptions of the fusion ideal for our fusion rings arising from quantum groups which makes it possible to connect them with existing descriptions in the literature (where $\ell$ is often replaced by the ``level"  $k = \ell - h$, respectively $k = \ell' - h'$ (see the case by case treatments below) and the dot action by the ordinary action).

The paper is organized as follows. In Section 2 we establish, based on \cite{AP}, the multiplication rules in the fusion ring $\mathcal F_q$ for $U_q(\mathfrak g)$ and give a general recipe how to produce a set of generators for the tensor ideal $\mathcal N_q$ (and hence for the fusion ideal $I_\ell$). We apply this for type $A$ in Section 3 to show how our formulas translate into the combinatorics in \cite{KorffStroppel} and demonstrate how our results easily provide minimal sets of generators for $I_\ell$. Then we introduce in Section 4 some analogous ``combinatorics" for type $C$ and in the even case we are again able to cut our set of generators down to a minimal such set. In Sections~5-7 we give presentations for the types $D$, $B$ and $G_2$, respectively.

\section{The fusion rule for $q$-groups at roots of unity} \label{s1}
In this section we recall the main definitions and facts about fusion
rings for quantum groups at complex roots of unity. Our presentation
is based on \cite {Atensor} and
\cite{AP}.

Suppose $\mathfrak g$ is a complex simple Lie algebra.
For an indeterminate $v$ denote by $U_v$ the quantum group over $\Q(v)$ corresponding
$\mathfrak g$. This is the $\Q(v)$-algebra with
generators $E_i, F_i, K_i^{\pm 1}$, $i = 1, 2, \dots , n=\mathrm{rank}(\mathfrak g)$
and relations as given e.g. in \cite[Chapter~5]{Ja}.

Set $A = \Z[v, v^{-1}]$. Then $A$ contains the quantum numbers
$[r]_d = \frac{v^{dr} - v^{-rd}}{v^d - v^{-d}}$ for any $r, d \in \Z$, $d \neq 0$, as
well as the corresponding $q$ binomials $\left[\genfrac{}{}{0pt}{}{m}{t}\right]_d$,
$m \in \Z, t \in \N$ defined via the factorials $[r]_d! = [r]_d[r-1]_d \cdots [1]_d$ for $r \geq 0$.
In case $d = 1$ they become the usual binomial numbers and so we will often omit the $d$ from the notation.

Let $C$ be the Cartan matrix associated with $\mathfrak g$. We denote by $D$ a diagonal
matrix whose entries are relatively prime natural numbers $d_i$ with the property that
$D C$ is symmetric. Then we set $E_i^{(r)} = E_i^r/[r]_{d_i}!$. With a similar expression
for $F_i^{(r)}$ we define now the $A$-form $U_A$ of $U_v$ to be the $A$-subalgebra of $U_v$
generated by the elements $E_i^{(r)}, F_i^{(r)}, K_i^{\pm 1}$, $i= 1, \dots, n$, $ r \geq 0$.
This is the Lusztig {\it divided power quantum group}.

Fix now a root of unity $q \in \C$ of order $\ell$.  The corresponding
quantum group is the specialization $U_q = U_A \otimes _A \C$ where $\C$ is considered
as an $A$-module via $v \mapsto q$, cf. \cite{Lu1}, \cite{Lu2}. We abuse notation and write
$E_i^{(r)}$ also for the element $E_i^{(r)} \otimes 1 \in U_q$ and similarly for
$F_i^{(r)}$.

We have a triangular decomposition $U_q = U_q^-U_q^0U_q^+$ with $U_q^-$ and $U_q^+$ being
the subalgebra generated by $F_i^{(r)}$ or $E_i^{(r)}$, $i = 1, \dots , n$,  $r \geq 0$,
respectively. The ``Cartan part'' $U_q^0$ is the subalgebra generated by $K_i^{\pm 1}$ and
$\left[\genfrac{}{}{0pt}{}{K_i}{t}\right]$, $i = 1, \dots , n, \; t \geq 0$, where
\begin{displaymath}
\left[\genfrac{}{}{0pt}{}{K_i}{t}\right] = \prod _{j=1}^t
\frac{K_iv^{d_i(1-j)} - K_iv^{-d_i(1-j)}}{v^{d_ij} - v^{-d_ij}}.
\end{displaymath}
We denote the ``Borel subalgebra'' $U_q^0U_q^+$ by $B_q$.

Recall that $U_v$ is a Hopf algebra, for explicit formulas for the comultiplication $\Delta$, counit $\epsilon$
and antipode $S$ see \cite[4.11]{Ja}. It is easy to see that their restrictions give $U_A$
the structure of a Hopf algebra over $A$. Then $U_q$ also gets an induced Hopf algebra structure.

Let $X$ be the lattice of integral weights with the positive cone
$X^+$ of integral dominant weights. Then each $\lambda \in X$ defines
in a natural way a $1$-dimensional $B_q$-module, also denoted $\la$,
and we set
\begin{equation}
\nabla_q(\lambda) = \Ind_{B_q}^{U_q} \lambda:=
F(\Hom_{B_q}(U_q,\lambda)),
\end{equation}
where $F$ is the functor which sends a $U_q$-module into its maximal
integrable submodule, cf. 2.8 in \cite{APW}.

Then $\nabla_q(\lambda) \neq 0$ iff $\lambda \in X^+$ and for such
$\lambda$ it contains a unique simple $U_q$-submodule which we denote
$L_q(\lambda)$.
The modules $L_q(\lambda)$, $\lambda \in X^+$ form a complete set of
pairwise non-isomorphic representatives for the isomorphism classes of
simple $U_q$-modules of type $\bf 1$.

\subsection{The category $\mathcal C_q$}
Let $\mathcal C_q$ be the category of finite dimensional
$U_q$-modules of type $\bf 1$, and denote by $\Gr(\mathcal C_q)$ its
Grothendieck group, i.e. the abelian group generated by the
isomorphism classes $[M]$ of objects $M$ in
$\mathcal C_q$ modulo the relation $[C]=[A]+[B]$ for any short exact
sequence $0 \rightarrow A\rightarrow C\rightarrow B \rightarrow 0$ in
$\mathcal C_q$.
The sets $\{[L_q(\lambda)]\}_{\lambda \in X^+}$
and $\{[\nabla_q(\lambda)]\}_{\lambda \in X^+}$ are two $\Z$-bases of
$\Gr(\mathcal C_q)$. For any element $f \in \Gr(\mathcal C_q)$ we
denote by $[f:\nabla_q(\lambda)]$ the coefficient of $f$ when
expressed in the basis $\{[\nabla_q(\lambda)]\}_{\lambda \in X^+}$,
i.e.
\begin{equation} f = \sum_{\lambda \in X^+} [f:\nabla_q(\lambda)]
[\nabla_q(\lambda)].
\end{equation}
If $f = [M]$ with $M \in \mathcal C_q$ we write
$[M:\nabla_q(\lambda)]$ instead of $[[M]:\nabla_q(\lambda)]$.
The comultiplication of $U_q$ turns $\mathcal C_q$ into a tensor
category and gives $\Gr(\mathcal C_q)$ a natural ring structure, the
{\it Grothendieck ring} of $\Gr(\mathcal C_q)$.

Special elements in $\Gr(\mathcal C_q)$ are the {\it Euler characters}
$\chi (N)$  of finite dimensional $B_q$-modules $N$ defined as
follows:
\begin{equation} \chi (N) = \sum_{i\geq 0} (-1)^i
[\mathcal{R}^i\Ind_{B_q}^{U_q} N],
\end{equation}
where $\mathcal{R}^i\Ind_{B_q}^{U_q}$ denotes the $i$th right derived
functor of the left exact functor $\Ind_{B_q}^{U_q}$. Then $\chi$ is
additive with respect to short exact sequences. Moreover, on the $1$-dimensional $B_q$-module
given by a character $\lambda \in X^+$ we have 
$\chi(\lambda) = [\nabla_q(\lambda)]$, by the
$q$-version of Kempf's vanishing theorem, \cite{RH}. More generally,
\begin{equation}
\label{altern}
\chi(\mu)=(-1)^{l(w)} \chi(w\cdot \mu)
\end{equation} for all $\mu \in X$ and $w \in W$. 

The $B_q$-modules we want to consider are finite dimensional and split
into weight spaces
$M = \oplus_\mu M_\mu$ as $U_q^0$-modules. The corresponding $B_q$-filtration of $M$
gives via the additivity of $\chi$ 
\begin{equation}
\chi (M) = \sum_\mu (\dim M_\mu) \chi(\mu).{\label {eq1}}
\end{equation}
The tensor identity $\mathcal{R}^i \Ind_{B_q}^{U_q} (M \otimes \mu)
\simeq M \otimes \mathcal{R}^i\Ind_{B_q}^{U_q} \mu $ for all $i \in \N$ and
$M \in \mathcal C_q$ implies
\begin{equation}
\chi (M \otimes \mu) = [M] \chi (\mu).
\end{equation}
The additivity of $\chi$ then gives for each $\lambda
\in X^+$
\begin{equation}
[M] [\nabla_q(\lambda)] = \chi (M \otimes \lambda) = \sum_{\nu \in X}
(\dim M_\nu) \chi(\lambda + \nu) {\label {eq3}}.
\end{equation}
Using \eqref{altern} we may rewrite this as
\begin{equation}
[M] [\nabla_q(\lambda)] = \sum_{\nu \in X^+} \left(\sum_{w\in W}
(-1)^{l(w)} \dim M_{w\cdot \nu - \lambda}\right) [\nabla_q(\nu)].
{\label {eq4}}
\end{equation}

This formula can then also be written as
\begin{equation}
[M \otimes \nabla_q(\lambda): \nabla_q(\nu)] = \sum_{w \in W}
(-1)^{l(w)}\dim M_{w \cdot \nu - \lambda}. {\label {eq5}}
\end{equation}
If all weights of $M\otimes \lambda$ are dominant after adding $\rho$,
then simplifies to
\begin{equation}
[M \otimes \nabla_q(\lambda): \nabla_q(\nu)] = \dim M_{\nu -
\lambda}.{\label {eq6}}
\end{equation}

The antipode on $U_q$ gives the linear dual $M^*$ of a $U_q$-module
$M$ a $U_q$-module structure. For each $\lambda \in X^+$ we define
$\Delta_q(\la)$ as the dual of
$\nabla_q(-w_0\la)$. Then $\Delta_q(\la)$ has the same character as
$\nabla_q(\lambda)$ and it has $L_q(\lambda)$ as its unique simple quotient
(it is the Weyl module with highest
weight $\lambda$).
A module $Q \in \mathcal C_q$ is called {\it tilting} if it has both a
$\nabla_q$- and a $\nabla_q$-filtration.
For each $\lambda \in X^+$ there exists a
unique indecomposable tilting module $T_q(\lambda)$ which has highest
weight $\lambda$. Moreover, the set $\{T_q(\lambda)\}_{\lambda \in
X^+}$ is, up to isomorphisms, a
complete list of indecomposable tilting modules in $\mathcal C_q$, see
e.g. \cite{Ringel}, \cite{Do}, \cite{Atensor}.

We denote by $\mathcal T_q$ the subcategory of $\mathcal C_q$
consisting of all tilting modules. This is an additive (but not
abelian) subcategory and since the collection of modules
with a $\nabla_q$-filtration is stable under tensor products, we see
that $\mathcal T_q$ is an additive tensor category, see \cite{Atensor}.

Set $K = \Gr(\mathcal T_q)$, the Grothendieck ring of $\mathcal T_q$. Since all 
short exact sequences consisting of tilting modules split, $K$ is defined
by
the relations $[Q] = [Q_1] + [Q_2]$ whenever $Q = Q_1 \oplus Q_2$ in
$\mathcal T_q$. The multiplicative structure on $K$ comes from the tensor product on $\mathcal T_q$.  
As a $\Z$-module $K$ is free with basis $\{[T_q(\lambda)]\}_{\lambda \in X^+}$. Note
that we again write $[Q]$ for the image in $K$ of a module $Q
\in \mathcal T_q$. We shall denote by $(f:T_q(\lambda))$ the
coefficient of $f \in K$
when expressed in the above basis. Thus if $Q \in \mathcal T_q$, then
$(Q:T_q(\lambda))$ is the multiplicity of $T_q(\lambda)$  as a direct
summand of $Q$.

We denote by $\mathcal{A}_\ell$ the {\it fundamental
alcove} in $X^+$, see \cite{Atensor} which is defined, excluding the special type $G_2$ treated later, as follows. When $\ell$ is odd  then
\begin{eqnarray*}
\mathcal{A}_\ell& = &\{\lambda \in X^+ \mid \langle \lambda + \rho,
\alpha_0^{\vee} \rangle < \ell  \}
\end{eqnarray*}
where $\alpha_0$ denotes the maximal short root. The closure of $\mathcal{A}_\ell$ is
\begin{equation*}
\bar{\mathcal{A}_\ell} = \{\lambda \in X \mid \langle
\lambda + \rho, \alpha_i^{\vee} \rangle \geq 0\text { for all simple roots } \alpha_i \text { and }  \langle
\lambda + \rho, \alpha_0^{\vee} \rangle \leq \ell \}.
\end{equation*}

If $\ell$ is even, then we set $\ell' = \ell/2$ and
\begin{eqnarray}
\mathcal{A}_\ell &=& \{\lambda \in X^+ \mid \langle \lambda + \rho, \beta_0^{\vee}
\rangle < \ell'  \},
\end{eqnarray}
where $\beta_0$ denotes the maximal (long) root. Let $\bar{\mathcal{A}_\ell}$ be its closure. In both the odd and even case the {\it affine Weyl group} $W_\ell$ is
the group generated by the reflections in the walls of $\mathcal{A}_\ell$. A
weight is called {\it $\ell$-singular} if it lies on a wall for
$W_\ell$. In case $G_2$ the shape of $\mathcal{A}_\ell$ depends additionally on whether  $3$ divides $\ell$ or not, see Section \ref{G2A}.

\subsubsection{Warning.} Note that in classical types for odd $\ell$ (respectively for $\ell$ prime to $3$ when the type is $G_2$) our affine Weyl group $W_\ell$ is actually the
affine Weyl group for the dual root system in the Bourbaki convention, cf. \cite[Chapter VI, \S2]{Bo}.
In the case of even $\ell$ (respectively for $\ell$ not prime to $3$ if the type is $G_2$) our affine Weyl group $W_\ell$ corresponds to the affine Weyl group
for the root system itself.
\vskip .5 cm
If $\lambda \in \bar{\mathcal{A}_\ell} \cap X^+$, then the linkage principle 
(\cite{AP}) gives
\begin{equation}L_q(\lambda) = \nabla_q(\lambda) = T_q(\lambda)
{\label {eq7}}.
\end{equation}
We have the following formula, cf \cite{AP}, \cite{GM}

\begin{thm}
Let $Q \in \mathcal T_q$ and $\lambda \in \mathcal{A}_\ell$. Then
$$(Q:T_q(\lambda)) = \sum_w (-1)^{l(w)} [Q:\nabla_q(w\cdot
\lambda)],$$
where the sum
runs over those $w \in W_\ell$ for which $w\cdot \lambda \in X^+$.
\end{thm}

\begin{proof}[{\bf Proof} (Sketch, cf. \cite{AP}):]
Since both sides are additive in $Q$ we may assume that $Q = T_q(\nu)$
for some $\nu \in X^+$. Then the left hand side is $\delta_{\lambda,
\nu}$. When $\nu = \lambda$ we see from
\eqref{eq7} that the right hand side is $1$. The linkage principle
\cite{AP} gives that it is $0$ unless $\nu \in W_\ell \cdot \lambda$
and so we are left to show that it vanishes when $\nu = w\cdot
\lambda$ for some non-trivial element $w\in W_\ell$. This follows from
the
inductive construction of $T_q(\nu)$ via translation functors, see
\cite{Atensor}.
\end{proof}

By the above we can apply this formula to $Q = L_q(\lambda) \otimes
L_q(\mu)$ when $\lambda, \mu \in \bar{\mathcal{A}_\ell}$ and obtain via formula
\eqref{eq5} the following {\it fusion rule}.

\begin{cor}
\label{Cor12}
Suppose $\lambda, \mu \in \bar{\mathcal{A}_\ell} \cap X^+$. Then for $\nu \in
\mathcal{A}_\ell$ we have
$$ (L_q(\lambda) \otimes L_q(\mu) : T_q(\nu)) = \sum_{w \in W_\ell}
(-1)^{l(w)} \dim \nabla_q(\lambda)_{w\cdot \nu - \mu} .$$
\end{cor}

Note that the right hand side is a known integer because the weight
spaces of $\nabla_q(\lambda)$ are determined e.g. by the Weyl
character formula.
A simple special case of this is the following (compare \eqref{eq6}).

\begin{cor}
\label{Cor13}
Suppose $\lambda, \mu \in \bar{\mathcal{A}_\ell} \cap X^+$ are such that $\eta +
\mu \in \bar{\mathcal{A}_\ell}$ for all weights $\eta$ of $\nabla_q(\lambda)$.
Then for any $\nu \in \mathcal{A}_\ell$ we have
$$ \left(L_q(\lambda) \otimes L_q(\mu) : T_q(\nu)\right) = \dim
\nabla_q(\lambda)_{\nu - \mu} .$$
\end{cor}

\subsection{The fusion rings  $\mathcal F_q=\mathcal
F_q(\mathfrak{g},\ell)$}
Let $\mathcal N_q$ denote the subcategory of $\mathcal T_q$ consisting
of those tilting modules which have no summands $T_q(\lambda)$ with
$\lambda \in \mathcal A_\ell$.
These are the {\it negligible tilting modules} which all have quantum
dimension $0$, cf. \cite[\S3]{AP}. It is also proved in \cite{AP} that
$\mathcal N_q$ is a tensor ideal in $\mathcal T_q$ and the quotient
$\mathcal T_q/\mathcal N_q$ is semisimple.

We define the {\it fusion ring} $\mathcal F_q=\mathcal
F_q(\mathfrak{g},\ell)$ to be the corresponding quotient of
$K$. For $\lambda \in \mathcal{A}_\ell$  we denote the image of
$[T_q(\lambda)]$ in
$\mathcal F_q$ by $[\lambda]$. Then the set $\{[\lambda]\}_{\lambda
\in \mathcal{A}_\ell}$ constitutes a $\Z$-basis for $\mathcal F_q$. Note that
$\mathcal F_q$ is a commutative,
associative ring. The commutativity follows from the fact that a
tilting module is determined, up to isomorphisms, by its character, so
that the tensor product on $\mathcal T_q$ is
commutative.

By Corollary \ref{Cor12} the multiplication in $\mathcal F_q$ is given
in basis vectors by
\begin{eqnarray}
[\lambda] [\mu] &=& \sum_{\nu \in \mathcal{A}_\ell} (L_q(\lambda) \otimes
L_q(\mu) : T_q(\nu)) [\nu]\nonumber\\
 &=&\sum_{\nu \in \mathcal{A}_\ell} (\sum_{w \in W_\ell} (-1)^{l(w)} \dim
 \nabla_q(\lambda)_{w\cdot \nu - \mu}) [\nu]{\label {eq8}}.
\end{eqnarray}

We extend the notation $[\lambda] \in \mathcal F_q$ to all $\lambda
\in X$ in the following way: If $\lambda \in X$ is $\ell$-singular,
then we set $[\lambda] = 0$. If $\lambda$ is
not $\ell$-singular, then there exists a unique $w \in W_\ell$ with
$w\cdot \lambda \in \mathcal{A}_\ell$ and we define $[\lambda] = (-1)^{l(w)}
[w\cdot \lambda]$. In this notation we can also
formulate \eqref{eq8} as 
\begin{equation}  [\lambda] [\mu] = \sum_{\nu \in X} \dim
L_q(\lambda)_\nu [\mu + \nu] {\label {eq9}}
\end{equation}
for all $\lambda, \mu \in \mathcal{A}_\ell$.

\subsubsection{Generators of $\mathcal N_q$}
Suppose $\omega_i \in \mathcal A_{\ell}$ for all $i$. For classical types this
means: If $\ell$ is odd, then $\ell \geq n+1$ for type $A_n$, $\ell
\geq 2n+1$ for types $B_n$ and $C_n$, and
$\ell \geq 2n-1$ for type $D_n$. If $\ell$ is even, then $\ell \geq
2n+2$ for types $A_n, B_n$ and $C_n$, whereas $\ell \geq 4n-2$ for
type $D_n$.

The linkage principle shows that this assumption implies that
$L_q(\omega_i) = T_q(\omega_i)$ for all $i$. Here the $\omega_i$'s are the fundamental weights in $X$.
Hence the
$[L_q(\omega_i)]$'s generate $K$.

Let us denote by
$\leq' $ the ordering on $X$ defined by $\la \leq' \mu$ iff $\mu =
\la + \sum_i m_i\omega_i$ for some $m_i \in \Z_{\geq 0}$. Note that this ordering is not
the same as the usual ordering $\leq$ on $X$ defined by the positive roots. We do have, however,
that if $\lambda \leq' \mu$, then $\mu - \lambda$ is a rational linear combination of positive roots
(with denominators at worst equal to the index of connection for our root system).

\begin{prop}
\label{generators}
The tensor ideal $\mathcal N_q$ in $\mathcal T_q$ is generated by the
set $\mathcal{G}=\{T_q(\mu) \mid \mu \text { minimal in } X^+\setminus
\mathcal{A}_\ell \text { with respect to } \leq' \}.$
\end{prop}

\begin{proof}
By definition of $\mathcal N_q$ it contains the ideal $I$ generated
by $\mathcal{G}$. To show the converse we need to check that
$T_q(\lambda)\in I$ for all $\la \in X^+\setminus \mathcal{A}_\ell$. To see this pick
$\gamma \in \sum_i \Z\alpha_i^\vee$ such that
$\langle \alpha_j, \gamma\rangle \in \Z_{>0}$ for all simple roots $\alpha_j$.
Multiplying $\gamma$ if necessary by a large enough integer we can ensure
that also $\langle\omega_j, \gamma\rangle \in \Z_{>0}$ for all $j$.
Suppose now that there exists $\lambda \in X^+\setminus \mathcal{A}_\ell$ with $T_q(\lambda) \notin I$.
Then we choose such a $\lambda$ for which $\langle \lambda, \gamma \rangle$ is minimal.
Now this $\lambda$ cannot be minimal with respect to the ordering $\leq'$ because then
$T_q(\lambda)$ would be one of the generators for $I$. Hence there exists $i$ such that $\lambda - \omega_i \in
X^+\setminus \mathcal{A}_\ell$. As $\langle \lambda - \omega_i, \gamma \rangle < \langle \lambda, \gamma \rangle$
our assumption on $\lambda$ implies that $T_q(\lambda - \omega_i) \in I$.
But then $I$ also contains $T_q(\la - \omega_i) \otimes
L_q(\omega_i)$. Now this tilting module decomposes as $T_q(\lambda)
\oplus T'$ with $T'=\bigoplus_{\nu} a_\nu T_q(\nu)$, where $a_\nu\not= 0$
implies $\nu<\la$ and $\nu\in X^+\setminus \mathcal{A}_\ell$. In particular
$T'\in I$ because all its summands $T_q(\nu)$ satisfy $\langle \nu, \gamma \rangle < \langle \lambda, \gamma \rangle$.
Hence we have a contradiction.
\end{proof}

\begin{remark}
In type $A$ we always have that the set of minimal elements in
$X^+\setminus \mathcal{A}_\ell$ are those which belong to the upper wall of
$\mathcal{A}_\ell$. Case by case considerations reveals that the
same is also true for type $C$ whereas for types $B$ and $D$ it is
only so when $\ell$ is odd.

\end{remark}

\section{Three descriptions in Type $A$}
We shall connect our fusion ring directly with the description of the
Verlinde algebra for $\hat{\mathfrak{sl}}_n$ at a fixed level $k$
given by Korff and the
second author, \cite{KorffStroppel}, \cite{W}, and extend the combinatorial
techniques developed there to the case
$\mathfrak{g}=\mathfrak{sp}_{2n}$ in the next section. Note
that these two fusion rings play a special role, since their explicit
ring structure is known and they are reduced complete intersection
rings, see \cite{BK}.

In this section we consider first $\mathfrak {gl}_n$ and then
$\mathfrak {sl}_n$. We assume for convenience that $\ell$ is odd
(except in the last subsection where we explain the
easy transition to the even case). The corresponding level $k$ is then
given by $k = \ell-n$.

\subsection{Fusion algebras for $\mathfrak {gl}_n$}
\label{sec:gl}
We shall now study the fusion ring in the case where $\mathfrak g =
\mathfrak {gl}_n$. In this case we identify $X=\Z^n$ by choosing the
standard basis $\{\epsilon_1, \epsilon_2, \cdots , \epsilon_n\}$ and
write $\la\in X$ as $\lambda = \sum_{i = 1}^n \lambda_i \epsilon_i$ or
alternatively
$\lambda = (\lambda_1, \lambda_2, \cdots ,\lambda_n)$. The positive
roots for $\mathfrak {gl}_n$ are $\epsilon_i - \epsilon_j$ with $1
\leq i < j \leq n$. Hence
$\lambda$ is dominant iff $\lambda_1 \geq \lambda_2 \geq \cdots \geq
\lambda_n$.

In this terminology the alcove $\mathcal{A}_\ell$ is determined by
\begin{equation}
\mathcal{A}_\ell = \{\lambda \in X^+ \mid \lambda_1 - \lambda_n \leq \ell -
n\}.
\end{equation}
In particular $\mathcal{A}_\ell = \emptyset $ if $\ell <n$. So we assume $0\leq
\ell-n=:k$ in the following and call $k$ the {\it level}.

We set $\omega_i = \epsilon_1 + \epsilon_2 + \dots + \epsilon_i$. This
is the $i$-th fundamental weight. The weights of $\nabla_q(\omega_i)$
are $\{ \epsilon_{j_1} +
\epsilon_{j_2} + \cdots + \epsilon_{j_i} \mid 1 \leq j_1 < j_2 <
\cdots < j_i \leq n\}$, all occurring with multiplicity $1$ (the
$\omega_i$'s are minuscule). In particular,
we observe that if $\mu \in \mathcal{A}_\ell$, then $\eta + \mu \in \bar{\mathcal{A}_\ell}$
for all weights $\eta$ of $\nabla_q(\omega_i)$. This means that in the
fusion ring $\mathcal F_q(\mathfrak {gl}_n,\ell)$ for
$U_q(\mathfrak {gl}_n)$ we have (cf. Corollary \ref{Cor13} or
\eqref{eq9}) the Pieri type rule
\begin{equation} [\omega_i] [\mu] = \sum [\epsilon_{j_1} +
\epsilon_{j_2} + \cdots + \epsilon_{j_i} + \mu], {\label {eq10}}
\end{equation}
where the sum runs over those $i$-tuples $1 \leq j_1 < j_2 < \cdots <
j_i \leq n$ for which $\epsilon_{j_1} + \epsilon_{j_2} + \cdots +
\epsilon_{j_i} + \mu \in \mathcal{A}_\ell$.

Let $0 \leq i \leq n-1$ and consider the operator ${\bf a}_i$ on
$\mathcal{A}_\ell$ given by
\begin{equation}
{\bf a}_i (\lambda) = \begin{cases} \lambda + \epsilon _{i+1}& \text {if }  \lambda + \epsilon _{i+1} \in \mathcal{A}_\ell, \\ 
0 &\text  {otherwise.}
 \end{cases}
\end{equation}
We may represent an arbitrary $\lambda = (\lambda_1, \lambda_2, \cdots
, \lambda_n) \in X$ by $n$ rows of boxes in which the $i$-th row is
infinite to the left and stops at the
number $\lambda_i$ (note that $\lambda_i$ may be negative). In this
configuration the rows are non-increasing and there are at most $k$
more boxes
in the first row than in the last. The operator ${\bf a}_0$ adds a box
in the first row if there are less than $k$ more boxes in this row
than in the last. Otherwise, ${\bf a}_0$
kills $\lambda$. Similarly, if $i>0$, then ${\bf a}_i$ adds a box in
the $i+1$'th row if this results in a non-increasing configuration and
maps to $0$ otherwise. The differences $m_i=\lambda_i-\lambda_{i-1}$
encode $\la$ in the basis of fundamental weights, $\la=\sum_i
m_i\omega_i$. Then $\lambda \in \mathcal{A}_\ell$ if and only if $m=\sum_{i=1}
^{n-1}m_i\leq k$. Setting $m_0=k-m$ as in
\cite{KorffStroppel}, weights can be viewed as a configuration of $k$
particles on a circle with $n$ marked points with $m_i$ particles at
place $i$ (viewed as an extended Dynkin diagram of type
$\tilde{A}_{n-1}$) and these operators can be thought of as {\it ``particle hopping" }from one point to the next in clockwise direction. If there
are no
particles at the point from where the operator makes a particle hop,
the operator kills the configuration instead:

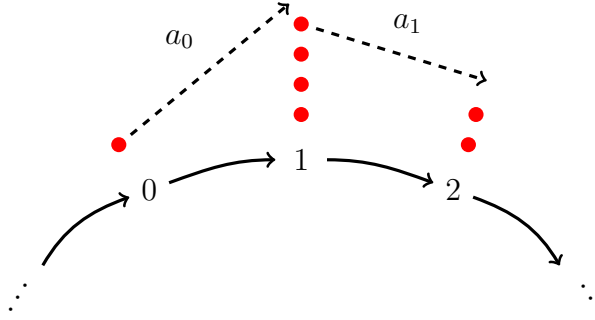
\begin{center}
   \begin{tikzpicture}[very thick,scale=2]
      \node[circle,inner
sep=2pt,outer sep=2pt] (z) at (0,0) {$1$};
      \node (n) at (-1,-.2) {$0$};
      \node (o) at (1,-.2) {$2$};
      \draw[->] (n) to[out=20,in=180] (z);
   \draw[->] (z) to[out=0,in=160] (o);
    \draw[->] (-1.7,-.7) to[out=60,in=-160] (n);
  \draw[->] (o) to[out=-20,in=120] (1.7,-.7);
 \node[rotate=120] at (1.9,-.9){$\cdots$} ;
\node[rotate=60] at (-1.85,-.9){$\cdots$ } ;
\node (rz0) at (0,1.1) {};
\node[circle,fill=red,inner sep=2pt,outer sep=3pt] (rz) at (0,.9) {};
\node[circle,fill=red,inner sep=2pt,outer sep=3pt] (rz1) at (0,.7)
{};
\node[circle,fill=red,inner sep=2pt,outer sep=3pt] (rz2) at (0,.5)
{};
\node[circle,fill=red,inner sep=2pt,outer sep=3pt] (rz3) at (0,.3)
{};
\node(ro) at (1.3,.5) {};
\node(a1) at (0.7,.9) {$a_1$};
\node[circle,fill=red,inner sep=2pt,outer sep=3pt] (ro1) at (1.15,.3)
{};
\node[circle,fill=red,inner sep=2pt,outer sep=3pt] (ro2) at (1.1,.1)
{};
\draw[->,dashed] (rz) to (ro) {};
\node[circle,fill=red,inner
sep=2pt,outer sep=3pt] (rn) at (-1.2,.1) {};
\draw[->,dashed] (rn) to (rz0) {};
\node(a0) at (-0.8,.8) {$a_0$};
\end{tikzpicture}
\end{center}
Here the action of ${\bf a}_1$ and ${\bf a}_0$ on
$\la=4\omega_1+2\omega_2$ is illustrated in the case $k = 7$.

Set  ${\underline{\bf a}} = ({\bf a}_0, {\bf a}_1, \cdots , {\bf
a}_{n-1})$ and define the non-commutative elementary symmetric
polynomials
${\bf e}_1({\underline{\bf a}}), {\bf e}_2({\underline{\bf a}}),
\cdots , {\bf e}_{n-1}({\underline{\bf a}})$ as in
\cite{KorffStroppel} by
\begin{equation}
{\bf e}_j({\underline{\bf a}}) = \sum_I \underline {\bf a}_I,
\end{equation}
where $I$ runs through all subsets of $\{0,1, \cdots , n-1\}$
consisting of exactly $j$ elements and $\underline {\bf a}_I$ is the
product
over $I$ of its elements in
anticlockwise cyclical order, see  \cite[\S 5]{KorffStroppel}.
By convention ${\bf e}_j({\underline{\bf a}}) = 0$ if either $j < 0$
or $j > n$ and ${\bf e}_0({\underline{\bf a}}) = {\bf
e}_n({\underline{\bf a}})$ is the identity.
Following \cite[\S6]{KorffStroppel} we define the {\it non-commutative
Schur polynomial}
\begin{equation}
\label{schur}
s_\lambda({\underline{\bf a}}) = \det ({\bf e}_{\lambda^t_i - i +
j}({\underline{\bf a}})),
\end{equation}
where $\la^t$ denotes the partition with $m_i$ rows of length $i$ and then 
the {\it combinatorial fusion ring} $\mathcal F_{comb}(\mathfrak
{gl}_n,\ell)$ as the free $\Z$-module with basis $\{\lambda\}_{\lambda
\in \mathcal{A}_\ell}$ equipped with the following multiplication
\begin{equation}
\label{fusion}
\lambda \star \mu = s_\lambda({\underline{\bf a}}) \mu.
\end{equation}
It is not obvious that the definition
\eqref{schur} makes sense and that \eqref{fusion} is commutative.
However the proof of the following result establishes this implicitly
(see \cite{KorffStroppel} for a completely different proof in that
setup using Bethe Ansatz techniques).

\begin{thm}
\label{thm:A}
The map $\Phi:\mathcal F_q(\mathfrak {gl}_n,\ell) \rightarrow
\mathcal F_{comb}(\mathfrak {gl}_n,\ell)$ taking each basis element
$[\lambda] \in \mathcal F_q(\mathfrak {gl}_n,\ell)$  to the basis
element $\lambda \in \mathcal
F_{comb}(\mathfrak {gl}_n,\ell)$  is a ring isomorphism.
\end{thm}

\begin{proof}[{\bf Proof:}]
We shall prove that $\Phi ([\lambda][\mu]) = \lambda \star \mu$ by
induction on $m = \lambda_1 - \lambda_n$. This is clear if $m = 0$. If
$m = 1$, then, up to a multiple
of $\epsilon_1 + \epsilon_2 + \cdots + \epsilon_n$, we have $\lambda
=
\omega_i$ for some $i$. In that case $s_{\omega_i}({\underline{\bf
a}}) = {\bf e}_i({\underline{\bf a}})$ operates by a formula identical
to \eqref{eq10}.
In particular, the operators ${\bf e}_i({\underline{\bf a}})$ commute,
since ${\bf e}_i({\underline{\bf a}})$ corresponds to multiplication
by $[\omega_i]$ in the commutative
fusion ring $\mathcal F_q(\mathfrak {gl}_n,\ell)$. Therefore it makes
sense to form the determinant in \eqref{schur}.

So suppose $m > 1$. We may write
$\lambda = \lambda' + \omega_i$ for some $i$ and some $\lambda'\in
\mathcal{A}_\ell$. Then we have
\begin{equation}
[\lambda'][\omega_i] = [\lambda] + \sum_{\eta} [\eta] {\label {eq11}}
\end{equation}
with the sum running over certain $\eta \in \mathcal{A}_\ell$ with $\eta_1 -
\eta_n < m$. Hence, by the above combined with the induction hypothesis,
we get
\begin{eqnarray*}
&\Phi([\lambda] [\mu])=\Phi( [\lambda'][\omega_i][\mu]) - \sum_\eta
\Phi([\eta][\mu])&\\
&= \Phi([\lambda'])\Phi([\omega_i][\mu]) - \sum_\eta \eta \star \mu
= \lambda' \star \omega_i \star \mu - \sum_\eta \eta \star \mu =
\lambda \star \mu.&
\end{eqnarray*}
Here the last equality comes from the fact that $\lambda' \star
\omega_i = \lambda +
\sum_\eta \eta$ where the sum ranges over the same $\eta$'s as in
\eqref{eq11}.
\end{proof}

\subsection{Fusion algebras for $\mathfrak {sl}_n$}
Preserving the notation from above we set $\epsilon = \epsilon_1 +
\epsilon_2 + \cdots + \epsilon_n$ and write $X_0 = \{ \lambda \in X
\mid \lambda_n = 0\}$. Then any $\lambda \in
X$ equals a unique element in $X_0$ modulo a multiple of $\epsilon$.
We set $X_0^+ = X_0 \cap X^+$ and $A_{0,\ell} = X_0 \cap \mathcal{A}_\ell$. Note
that $A_{0,\ell}$ is finite.

Restriction from $\mathfrak  {gl}_n$ to $\mathfrak  {sl}_n$ gives a
natural surjection from the fusion ring $\mathcal F_q(\mathfrak
{gl}_n,\ell)$ to the
corresponding fusion ring $\overline{\mathcal F_q(\mathfrak
{gl}_n,\ell)}=\mathcal F_q(\mathfrak {sl}_n,\ell)$ for $\mathfrak
{sl}_n$. Similarly, we get a surjection from
${\mathcal F_{comb}(n, \ell)}$ to the corresponding
construction $\overline {\mathcal F_{comb}(\mathfrak {gl}_n,\ell)}$
for $A_{0,\ell}$. We obtain

\begin{thm}
We have ring isomorphisms $$\overline{\mathcal F_q(\mathfrak
{gl}_n,\ell)} \simeq \mathcal F_q (\mathfrak {gl}_n,\ell)/([\epsilon]
-1)$$ and
$$\overline {\mathcal F_{comb}(\mathfrak {gl}_n,\ell)} \simeq {\mathcal
F_{comb}(\mathfrak {gl}_n,\ell)}/(\epsilon -1),$$ so that the
isomorphism $\Phi$ from Theorem \ref{thm:A} induces an isomorphism
$$\overline{\mathcal F_q(\mathfrak {gl}_n,\ell)}
\simeq \overline {\mathcal F_{comb}(\mathfrak {gl}_n,\ell)}.$$
\end{thm}

\subsubsection{$\ell$ even}
In case $\ell$ is even we set $\ell' = \ell/2$. Since the roots in
type $A$ have the same length the only change we have to make when
describing the fusion rules in this
case is to replace $\ell$ by $\ell'$, see \cite{AP}. We can then argue
exactly as before to see that all the above statements remain true.
Note in
particular that we need $\ell \geq 2n$ now as otherwise the
fundamental alcove $A_{\ell'}$ is empty.

\subsection{Commutative presentation}
The following connects our approach with well-known presentations of the fusion ring, \cite{BK}, \cite{BR}, \cite{Dou}.

\begin{thm}
There are isomorphisms of commutative rings
\begin{eqnarray*}
\mathcal F_{comb}(\mathfrak {gl}_n,\ell)&\simeq&\mZ[\chi({\omega_1}),\ldots,\chi({\omega_{n-1}})]/I\\
&\simeq& \mZ[\chi({\omega_1}),\ldots,\chi({\omega_{n-1}})]/J
\end{eqnarray*}
with $I=\langle \chi(s\omega_1)\mid k+1 \leq s < k+n \rangle$, $J=\langle \chi({k\omega_1+\omega_i})\mid 1\leq i\leq n-1 \rangle$.
\end{thm}

\begin{proof}
The generators given by Proposition \ref{generators} of the defining
ideal are the $\chi(\lambda)$'s with $\lambda = \sum_{i=1}^{n-1} m_i \omega_i$
satisfying $\sum_{i=1}^{n-1} m_i = k+1$. For such $\lambda$'s the Jacobi-Trudy
identity \cite[(A.5)]{FH} expresses $\chi(\lambda)$ as the determinant of
a matrix whose first row contains entries equal to the complete symmetric
polynomials $\chi((k+1)\omega_1), \cdots , \chi((k+n-1)\omega_1)$. Hence expanding
the determinant along this row we obtain $\chi(\lambda)$ as a
$\Z[\chi(\omega_1), \cdots , \chi(\omega_{n-1})]$-linear combination of
$\chi(s\omega_1), s= k+1, \cdots , k+n-1$. This proves that our defining ideal is
contained in the ideal $I$. The reversed inclusion is clear once we observe that for $s$ in the given range we have
$L_q(s\omega_1) = \nabla_q(s\omega_1) = T_q(s\omega_1)$, see \cite[Theorem 3.5] {T}.
For the induction step we use the Pieri formula \cite[(A.7)]{FH}
\begin{eqnarray*}
\chi(r\omega_1) \chi(\omega_i) &= &\chi(r\omega_1 + \omega_i) + \chi((r-1)\omega_1 + \omega_{i+1})
\end{eqnarray*}
for $r\geq 0$ and $1\leq i\leq n$. This gives for $k+1\leq s\leq k+n-2$
\begin{eqnarray*}
\chi(s\omega_1)\chi(\omega_1) &= &\chi((s+1)\omega_1) + \chi((s-1)\omega_1 + \omega_{2})\\
\chi((s-1)\omega_1)\chi(\omega_2) &= &\chi((s-1)\omega_1+\omega_2) + \chi((s-2)\omega_1 + \omega_{3})\\
&\vdots&\\
\chi((k+1)\omega_1)\chi(\omega_{s-k}) &= &\chi((k+1)\omega_1+\omega_{s-k}) + \chi(k\omega_1 + \omega_{s-k+1}).
\end{eqnarray*}
Reading from bottom to top we see that the outer terms are in $J$, hence the middle term as well and the claim follows.
\end{proof}

\section{Three descriptions in Type $C$}
In this section we assume that $\mathfrak g = \mathfrak{sp}_{2n}$.
\subsection{The classical case}
We still have $X = \Z^n$ with basis $\epsilon_1, \epsilon_2, \cdots ,
\epsilon_n$. The positive roots are now
$\{\epsilon_i - \epsilon_j, \epsilon_i + \epsilon_j \mid 1 \leq i < j
\leq n\} \cup \{2\epsilon_i \mid 1 \leq i \leq n\}$ and the simple
roots are $\alpha_i = \epsilon_i -
\epsilon_{i+1}, i = 1, \cdots , n-1$ together with $\alpha_n = 2
\epsilon_n$. The corresponding fundamental weights are $\omega_i =
\epsilon_1 + \epsilon_2 + \cdots + \epsilon_i$,
$i = 1, \cdots , n$. A weight $\lambda \in X$ can be expressed in
both the $\epsilon$-basis, say $\lambda = \sum_{i=1}^n \lambda_i
\epsilon_i$, and in the $\omega$-basis, say $\lambda =
\sum_{i=1}^n m_i\omega_i$. The dictionary between the $\lambda_i$'s
and the $m_i$'s is $m_i = \lambda_i - \lambda_{i+1}$, $i = 1, \cdots ,
n$ (with the convention that
$\lambda_{n+1} = 0$) and $\lambda_i = m_i + m_{i+1} + \cdots + m_n$.
We have that $\lambda$ is dominant iff $\lambda_1 \geq \lambda_2 \geq
\cdots \geq \lambda_n \geq 0$ iff $m_i
\geq 0$ for all $i$.

The highest short and and the highest long root are respectively
\begin{eqnarray*}
\alpha_0&=&\alpha_1 + 2 \alpha_2 + 2 \alpha_3 + \cdots + 2
\alpha_{n-1} + \alpha_n = \epsilon_1 + \epsilon_2\\
\beta_0&=&2\alpha_1 + 2 \alpha_2 + \cdots 2 \alpha_{n-1} +\alpha_n = 2
\epsilon_1.
\end{eqnarray*}
We have $\langle \omega_1, \alpha_0^\vee\rangle = 1$ whereas $\langle
\omega_i, \alpha_0^\vee\rangle = 2$
for $i>1$. On the other hand, $\langle \omega_i, \beta_0^\vee\rangle =
1$ for all $i$. Hence $\langle \rho, \alpha_0^\vee \rangle = 2n-1$ and
$\langle \rho, \beta_0^\vee \rangle =
n$.

Let now $V$ denote the natural $2n$-dimensional module for
$\mathfrak{sp}_{2n}(\C)$ with its standard basis $v_1, \cdots ,
v_{2n}$. Setting $\epsilon_{n+i} = -\epsilon_i, \;
i= 1, 2, \cdots , n$ we
have that $v_i$ has weight $\epsilon_i$, $i= 1, 2, \cdots , 2n$.

Letting $L_\C(\lambda)$ denote the irreducible $\mathfrak
{sp}_{2n}(\C)$-module with highest weight $\lambda \in X^+$ and let
$\chi (\la)$ be its character. We have clearly $L_\C(0) = \C$ and
$L_\C(\omega_1) = V$. If
$i>1$, then
$L_\C(\omega_i)$ fits into a short exact sequence
\begin{equation}
\label{ses}
 0 \to L_\C(\omega_i) \to \Lambda^iV \stackrel{f}\to \Lambda^{i-2}V
 \to 0,
\end{equation}
where $f$ comes from the symplectic form on $V$, \cite[Theorem
17.5]{FH}.

In analogy with the $\mathfrak{gl}_n$-case in Section \ref{sec:gl},
the weights of $\Lambda^i V$ are $\epsilon_J = \epsilon_{j_1} +
\epsilon_{j_2} + \cdots + \epsilon_{j_i}$ with $J = j_1 < j_2 < \cdots
< j_i$ running
through all increasing sequences in $\{1, 2, \cdots , 2n\}$ consisting
of $i$ elements. Note however, that since $\epsilon_{n+i} =
-\epsilon_i$ two different such sequences may
well lead to the same weight. For instance when $i=2$
the zero weight equals $\epsilon_{J_i}$ for all $J_i = \{i, n+i\}$,
$i= 1, 2, \cdots , n$ and so also $\epsilon_J = \epsilon_{J \cup J_i}$
whenever $J \cap J_i = \emptyset$.

\subsection{The quantum case}

Let now $U_q = U_q(\mathfrak {sp}_{2n})$ denote the quantum group for
$\mathfrak {sp}_{2n}$ at a complex root of unity $q$. We denote by
$\ell$ the order of $q$ and assume $\ell > 2n$. When $\ell$ is odd we have
\begin{eqnarray*}
\mathcal{A}_\ell &=& \{\lambda \in X^+ \mid \langle \lambda + \rho,
\alpha_0^\vee\rangle < \ell \} \\
&=& \{\lambda \in X^+ \mid m_1 + 2m_2 + 2 m_3 + \cdots + 2m_n <
\ell-2n+1\}
\end{eqnarray*}
In the case when $\ell$ is even we set  $\ell' = \ell/2$ and have
\begin{eqnarray*}
\mathcal{A}_\ell &=&
\{\lambda \in X^+ \mid \langle \lambda + \rho, \beta_0^\vee\rangle <
\ell' \}\\
&=& \{\lambda \in X^+ \mid m_1 + m_2+ \cdots + m_n < \ell'-n\}.
\end{eqnarray*}
In either case we have $\mathcal{A}_\ell \neq \emptyset$ is equivalent to our
assumption $\ell > 2n$.

Recall that for each $\lambda \in X^+$ the dual Weyl module
$\nabla_q(\lambda)$ is a module for $U_q$ whose weights are the same
as for the classical module $L_{\C}(\lambda)$. So if
for $i \leq n$ we set
\begin{equation}
V_q^i =
\begin{cases}
\;\nabla_q(0) \oplus \nabla_q(\omega_2) \oplus \cdots \oplus
\nabla_q(\omega_{i}) &\text { if $i$ is even};\\
\nabla_q(\omega_1) \oplus \nabla_q(\omega_3) \oplus \cdots \oplus
\nabla_q(\omega_{i}) &\text { if $i$ is odd};
\end{cases}
\end{equation}
then by \eqref{ses} $V_q^i$ has the same weights as $\Lambda^iV$.

Our assumptions on $\ell$ and $\ell'$ ensure that all the $\nabla_q
(\omega_i)$ are irreducible tilting modules because all $\omega_i$
belong to the closure of $\mathcal{A}_\ell$. In the following
we write therefore $L_q(\omega_i)$ instead of $\nabla_q(\omega_i)$.
The $V_q^i$ are therefore also tilting modules and formula \eqref{eq9}
gives
\begin{equation}
[V_q^i] [\mu] = \sum_{|J| = i} [\mu + \epsilon_J],
\end{equation}
where the $J$'s appearing in the sum are subsets of $\{1, 2, \cdots ,
2n\}$.
The formulas \eqref{ses} then imply again Pieri rules
\begin{equation}
\label{fusionC}
[\omega_i][\mu] = \begin{cases} \sum_{j=1}^{2n} [\mu + \epsilon_j],
&\text { if } i = 1, \\
\sum_{|J| = i} [\mu + \epsilon_J] - \sum_{|J| = i-2} [\mu +
\epsilon_J],& \text { if } i \geq
2.
\end{cases}
\end{equation}

\subsection{Combinatorics}

We shall now describe the product formulas \eqref{fusionC} combinatorially, in 
analogy with the type $A$ case considered in Section \ref{sec:gl}.
This will then allow us to deduce the general fusion rule in
terms of combinatorially described operators.

Consider the free $\Z$-modules $\Z[X]$ and $\Z[\mathcal{A}_\ell]$ with bases
$e^\la$ for  $\la\in X$ and $\la \in \mathcal{A}_\ell$, respectively.
We define a $\Z$-linear map $\pi_\ell : \Z[X] \to \Z[\mathcal{A}_\ell]$ by the
recipe
\begin{eqnarray*}
\pi_\ell(e^\lambda)& = &\begin{cases} (-1)^{l(w)} e^{w \cdot \lambda},
&\text { if there exists } w \in W_\ell \text { with }  w\cdot \lambda
\in \mathcal{A}_\ell, \\
0, &\text { otherwise.}
\end{cases}
\end{eqnarray*}
For $j \in \{1, 2, \cdots , 2n\}$ we define $\Z$-linear endomorphisms
${\bf a}_j$ of $\Z[X]$ by
\begin{equation}
{\bf a}_j(e^\lambda) = e^{\lambda + \epsilon_j},\; \lambda \in X.
\end{equation}
For each subset $J = \{j_1 < j_2 < \cdots <j_i\}  \subset \{1, 2, \cdots ,
2n\}$ these operators give rise to the operator ${\bf a}_J = {\bf
a}_{j_1} \circ {\bf a}_{j_2} \circ \cdots \circ {\bf a}_{j_i}$ on
$\Z[X]$. Then for any $i \leq n$ we define ${\bf e}_i : \Z[\mathcal{A}_\ell] \to
\Z[\mathcal{A}_\ell]$ by
\begin{equation}
\label{Defe}
{{\bf e}_i} (e^\lambda) = \sum_{|J| = i} \pi_\ell({\bf
a}_J(e^\lambda)),\; \lambda \in \mathcal{A}_\ell.
\end{equation}

\subsubsection{Case $\ell$ even}
Let $\ell$ be even. If $\lambda = \sum_{i=1}^n m_i \omega_i \in
X$ we set $m_0 = \ell'-n-1-\sum_{i=1}^n m_i$ and we identify $\lambda$
with the particle
configuration on the extended Dynkin diagram of type $\tilde{C}_{n}$ with $m_i$ particles if $m_i \geq 0$, respectively
$-m_i$ antiparticles if $m_i <0$, placed at node $i$, $i = 0, 1,
\cdots , n$. Counted with signs ($+$ for
particles and $-$ for antiparticles) we have a total
of $k:=\ell'-n-1$ particles.  This number $k$ is preserved by the
above operators ${\bf a}_i$ and called the {\it level}. For $ 1 \leq i
\leq n$
we may describe ${\bf a}_i$  as particle hopping from node $i-1$ to
node $i$. If there are no particles at node $i-1$, then we first add a
pair consisting of a particle and
an antiparticle to this node and then
make this particle hop. If at node $i$ we have only antiparticles, then the
particle added to this node annihilates one of the antiparticles. The
inverse ${\bf a}_{n+i}$ is hopping in the reverse direction:
\begin{equation}\label{hoppingCeven}
\xymatrix{\stackrel{0}\bullet
\ar@/^/[r]^{{\bf a}_1}&\stackrel{1}\bullet\ar@/^/[r]^{{\bf a}_2}
\ar@/^/[l]^{{\bf a}_{n+1}} & \stackrel{2}\bullet \ar@/^/[r]^{{\bf
a}_3}
\ar@/^/[l]^{{\bf a}_{n+2}}&\stackrel{3}\bullet
\ar@/^/[l]^{{\bf a}_{n+3}}& \cdots & \stackrel{n-1}\bullet
\ar@/^/[r]^{{\bf a}_n} & \stackrel{n}\bullet
\ar@/^/[l]^{{\bf a}_{2n}}}
\end{equation}

Let $\lambda \in \mathcal{A}_\ell$. This means that the particle configuration
corresponding to $\lambda$ contains no antiparticles. There are
$k=l'-n-1$  particles in this
configuration. Let $J$ be a subset of $\{1, 2, \cdots , 2n\}$ with
$|J| \leq n$. Note that $\langle \epsilon_J, \alpha_i^{\vee} \rangle
\in \{0, \pm 1, \pm 2\}$ for all $i
= 1, \cdots , n-1$ whereas  $\langle \epsilon_J, \alpha_n^{\vee}
\rangle,  \langle \epsilon_J, \beta_0^{\vee} \rangle\in \{0, \pm 1\}$.
Hence  $\lambda + \epsilon_J$ contains at most $2$ antiparticles at a
given node. This node cannot be an end node ($0$ or $n$) and if node
$i$ contains $2$ antiparticles,
then no antiparticles are positioned at the adjacent nodes $i-1$ and
$i+1$. When we apply $\pi_\ell$ we get $0$ if there is exactly $1$
antiparticle at some node, because then
the weight is $\ell$-singular. If there
are $2$ antiparticles at node $i$, then we replace the configuration by
minus the one where we have removed $1$ particle from each of the two
nodes $i-1$ and $i+1$ and placed both
of them at
node $i$ (thus annihilating the $2$ antiparticles there). If the
resulting configuration still contains antiparticles, we repeat the
above process. The end result corresponds then
to $\pi_\ell(e^{\lambda + \epsilon_J})$. Note that there are at most
as many annihilation steps as there are nodes with $2$ antiparticles.

The action of $W_\ell$ which is applied to remove 2 antiparticles can
be visualized on the extended Dynkin diagram of type $\tilde{C}_{n}$:
\begin{equation}\label{WeylCaffine}
\xymatrix{\stackrel{0}\bullet
\ar@{=}[r]|{\rangle}&
\stackrel{1}\bullet
\ar@{-}[r]&
\stackrel{2}\bullet
\ar@{-}[r]&
\stackrel{3}\bullet
&
\cdots&
\stackrel{n-1}\bullet
\ar@{=}[r]|{\langle}&
\stackrel{n}\bullet}
\end{equation}
Note that particles are always taken from neighboring vertices, one for each
edge; for the special vertices $0$ and $n$ two particles get moved
along a double edge against the direction of the arrow.
\subsubsection{Case $\ell$ odd}
Suppose that $\ell$ is odd. If  $\lambda = \sum_{i=1}^n
m_i\omega_i \in X$, we set this time $m_0 = \ell-2n - m_1 -2m_2 -2m_3 -
\cdots - 2m_n$.
We may again interpret
$\lambda$ as a configuration of particles and antiparticles placed at
the nodes $0, 1, \cdots, n$. If $m_i \geq 0$, there are $m_i$
particles at node $i$ and
if $m_i <0$, there are $-m_i$ antiparticles at node $i$. This time we
count particles and antiparticles at the nodes $2, 3, \cdots , n$
double, i.e. we have a total of $k=\ell-2n$ particles.
This
number is preserved by the ${\bf a}_i$'s which we can describe as
follows: For $i \leq n$ the operator ${\bf a}_i$ moves a particle from
node $i-1$ to node $i$ if $i \neq 2$, while
${\bf a}_2$ removes $1$ particle from each of the nodes $0$ and $1$
and fuse these into $1$ particle placed at node $2$. Again ${\bf
a}_{n+i}$ is the inverse of ${\bf a}_i$. So for instance
${\bf a}_{n+2}$ removes a particle from node $2$, splits it into $2$
and places these at nodes $0$ and $1$. Hence it could be thought of as
moving particles one step to
the left resp. right along the diagram \eqref{WeylBaffine} below
(except for $\bf a_1$ and $\bf a_{n+1}$ which moves between nodes $0$
and $1$).

Let $\lambda \in \mathcal{A}_\ell$. This means that the corresponding particle
configuration contains no antiparticles and $k$ particles. For any $J
\subset \{1, 2, \cdots , 2n\}$
the configuration for $\lambda + \epsilon_J$ contains at most $1$
antiparticle at node $n$ and at most
$2$ antiparticles at any other node.
Moreover, if node $i$ has 2 antiparticles, there are no antiparticles
at nodes $i-1$ and $i+1$. If there is a node with exactly $1$
antiparticle, then $\pi_\ell(e^{\lambda +
\epsilon_J}) = 0$, since the corresponding weight is $\ell$-singular.
If there are $2$ antiparticles at node $0$ or $1$ take a particle at
node $2$ (which counts double) and make it annihilate the $2$
antiparticles at the given
node. If there are $2$ antiparticles at node $2$, we annihilate them by
taking a particle from each of the nodes $0, 1$, and $3$ (the $2$
particles from nodes $0$ and $1$ need to combine in
order to annihilate one of the antiparticles from node $2$). Finally,
$2$ antiparticles at any node $i$ with $2<i<n$ are annihilated by
taking a particle from each of the two
adjacent nodes. If after this process the result still contains
antiparticles, we repeat it until this is no longer the case. This end
configuration is then, up to signs,
equal to $\pi_\ell(e^{\lambda + \epsilon_J})$. In case it is non-zero,
the sign is positive if we have used an even number of annihilation
steps and negative otherwise. Again the
number of steps is at most equal to the number of nodes with $2$
antiparticles.

In this case the action of $W_\ell$ which is applied to remove 2
antiparticles can be visualized using the extended Dynkin diagram of
type  $\tilde B_n^t$:
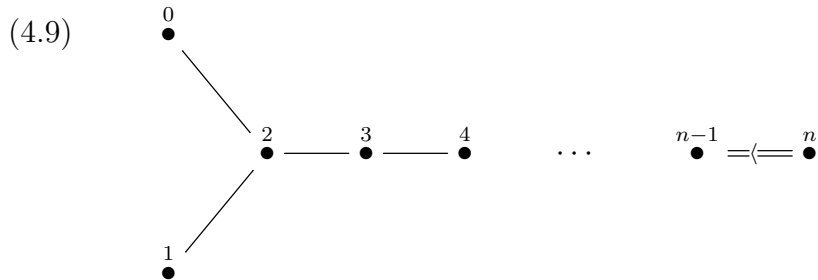
\begin{equation}
\label{WeylBaffine}
\xymatrix{\stackrel{0}\bullet
\ar@{-}[dr]&\\
&\stackrel{2}\bullet
\ar@{-}[r]&
\stackrel{3}\bullet
\ar@{-}[r]&
\stackrel{4}\bullet
&
\cdots&
\stackrel{n-1}\bullet
\ar@{=}[r]|{\langle}&
\stackrel{n}\bullet\\
\stackrel{1}\bullet
\ar@{-}[ur]}
\end{equation}
Again particles are always taken from neighboring vertices, one for each
single edge; for the special vertex $n$ two particles get moved along
a double edge against the arrow.

\subsection{Combinatorial fusion rule}

Let ${\bf e}_i$ be the operators on $\Z[\mathcal{A}_\ell]$ defined above. The
definitions of $\mathcal{A}_\ell$ and of ${\bf e}_i$ are different in the odd
and even cases. However, in both cases we have
\begin{prop}
The operators ${\bf e}_i$ and ${\bf e}_j$ commute for $1\leq i, j \leq
n$.
\end{prop}

\begin{proof} [{\bf Proof:}]
Note that as $\Z$-modules we have an isomorphism $\mathcal F_q \simeq
\Z[\mathcal{A}_\ell]$ such that the operator ${\bf e}_i$ corresponds to
multiplication by $V_q^i$ in $\mathcal F_q$. But $\mathcal F_q$ is
a commutative and associative ring, therefore multiplications by
$V_q^i$ and $V_q^j$ commute.
\end{proof}

Set now ${\bf e}'_0 = 1$, ${\bf e}'_1 = {\bf e}_1$ and ${\bf e}'_i =
{\bf e}_i -{\bf e}_{i-2}$ for $2\leq i \leq n$. If $\lambda \in
\mathcal{A}_\ell$ we define in analogy with (3.5) the operator ${\bf
s}'_\lambda$ on $\Z[\mathcal{A}_\ell]$ by
\begin{equation}
{\bf s}'_\lambda  = \det({\bf e}'_{\lambda^t_i-i+j}).
\end{equation}
By the above proposition the ${\bf e}'_i$'s clearly commute so that it
makes sense to form this determinant. By \cite[Appendix A.3, Corollary
24.24]{FH} the operator ${\bf s}'_\lambda$ corresponds to
multiplication by
$L_q(\lambda)$ in $\mathcal F_q$. Hence we define multiplication on
$\Z[\mathcal{A}_\ell]$ (as before we denote the basis vector corresponding to 
$\lambda \in \mathcal A_\ell$ by $\lambda$) as follows
\begin{equation}
\lambda \star \mu = {\bf s}'_\lambda(e^\mu)
\end{equation}
and denote the resulting ring by $\mathcal
F_{comb}(\mathfrak{sp}_{2n},\ell)$. This is then a commutative,
associative ring with unit $1=e^0$ and underlying vector space
canonically identified with $\mathcal F_q(\mathfrak{sp}_{2n},\ell)$
via $\la\mapsto [\la]$. Moreover, we obtain
\begin{thm}
\label{springs}
For any $\ell > 2n$ we have an isomorphism of rings $$\mathcal
F_q(\mathfrak{sp}_{2n},\ell)=\mathcal F_{comb}(\mathfrak{sp}_{2n},
\ell)$$
taking each basis element $[\lambda]$ to the basis
element $\lambda\in \mathcal F_{comb}(\mathfrak{sp}_{2n},
\ell).$
\end{thm}

\subsection{Commutative presentation}
Finally we connect our combinatorial description with well-known
presentations of the fusion ring, \cite{Gepner}, \cite{BR}. For the
analogous statement in type $A$ see \cite[Theorem 6.20]{KorffStroppel}.

\begin{thm}
Let $\ell$ be even. There is an isomorphism of commutative associative
rings
\begin{eqnarray*}
\Psi: \mathcal F_{comb}(\mathfrak{sp}_{2n}, \ell) \simeq
\mZ[\chi({\omega_1}),\ldots,\chi(\omega_n)]/
\langle
\chi({k\omega_1+\omega_i})\mid 1\leq i\leq n \rangle
\end{eqnarray*}
which sends $\la$ to the character $\chi(\la)$.
\end{thm}

Note that the number of generators in the ideal is independent of the
level in this case.

\begin{proof}
The fusion ring is a quotient of
$\mZ[\chi({\omega_1}),\ldots,\chi({\omega_n})]$ by some ideal $I_\ell$
whose generators are given by Proposition \ref{generators}. These are
the $[T_q(\la)]$ with $\la$ minimal such that $\langle
\la,\beta_0^\vee\rangle>k$. In our case this is precisely  $\langle
\la,\beta_0^\vee\rangle=k+1$, hence $\la_1=k+1$. Since
$[T_q(\la)]=[L_q(\la)]$ by \eqref{eq7}, Theorem \ref{springs} implies
that $I_\ell$ is generated by all $\chi(\la)$ where $\la_1=k+1$. Now
using the Jacobi-Trudi type determinant formula
\cite[3.9]{Krattenthaleretco} for symplectic characters, expanded
along the first row, implies that
$\chi(\lambda)$ is a
$\mZ[\chi({\omega_1}),\ldots,\chi({\omega_n})]$-linear combination of
the elements $\chi({(k+1)\omega_1})$ and
$\chi({(k+i+1)\omega_1}+\chi(k+1-i)\omega_1)$ for $1\leq i\leq n-1$.
Hence $I_\ell$ is contained in the ideal generated by these and they
are obviously also contained in this ideal. Expanding the determinant
formulas for $\chi({k\omega_1+\omega_i})$, we get
$\chi({k\omega_1+\omega_i})=\mp \chi({(k+i)\omega_1})+R$ for some linear
combination of $\chi(\lambda)$'s with $\la$ smaller in the dominance
ordering of partitions.
Therefore, our ideal is also generated by the $
\chi({k\omega_1+\omega_i})$ for $1\leq i\leq n$.
\end{proof}

For odd $\ell$ the presentation is not as explicit, but at least we
have

\begin{thm}
Let $\ell$ be odd. There is an isomorphism of commutative, associative
rings
\begin{eqnarray*}
\Psi: \mathcal F_{comb}(\mathfrak{sp}_{2n}, \ell) \simeq
\mZ[\chi({\omega_1}),\ldots,\chi({\omega_n})]/I_\ell
\end{eqnarray*}
which sends $\la$ to the character $\chi(\la)$. Here $I_\ell$ is
generated by all characters $\chi(\la)$ where $\la_1+\la_2=k+1$.
\end{thm}
\begin{proof}
We again use Proposition \ref{generators} which shows that $I_\ell$ is
generated by all $[T(\la)]$ where $\la$ satisfies
\begin{enumerate}
\item $m_1+2\sum_{i=2}^n m_i=k+1=\ell-2n + 1$ or

\item $m_1 = 0$ and $2\sum_{i=2}^n m_i=k+2=\ell-2n+2$.
\end{enumerate}
Note that $k$, and then also $k+2$, is odd, hence (2) has no solutions.
For the first set of generators we have again by \eqref{eq7} the
equality $[T_q(\la)]=[L_q(\la)]$ and $\la_1+\la_2=k+1$.
\end{proof}

\section{Even and odd fusion rings for Type D}
In this section we assume that $\mathfrak g = \mathfrak{so}_{2n}$. The simple roots are now $\alpha_i = \epsilon_i -
\epsilon_{i+1}, i = 1, \cdots , n-1$ together with $\alpha_n =
\epsilon_{n-1}+\epsilon_n$. The corresponding fundamental weights are
$\omega_i = \epsilon_1 + \epsilon_2 + \cdots + \epsilon_i$,
$i = 1, \cdots , n-2$ and $\omega_{n-1}=\frac{1}{2}
(\epsilon_1+\cdots\epsilon_{n-1}-\epsilon_n)$ and
$\omega_{n}=\frac{1}{2} (\epsilon_1+\cdots\epsilon_{n-1}+\epsilon_n)$.
The highest root is
\begin{eqnarray*}
\alpha_0&=&\alpha_1 + 2 \alpha_2 + 2 \alpha_3 + \cdots +  \alpha_{n-1}
+ \alpha_n = \epsilon_1 + \epsilon_2=\omega_2.
\end{eqnarray*}

Set $k=\ell-(2n-2)$ in case $\ell$ is odd and $\frac{\ell}{2}-(2n-2)$
if $\ell$ is even. Then $\mathcal{A}_\ell$ consists of all dominant weights
$\sum_{i=1}^n m_i\omega_i$ such that
\begin{eqnarray*}
m_1+2m_2+\cdots+ 2m_{n-2}+m_{n-1}+m_{n}&\leq&k.
\end{eqnarray*}
Proposition \ref{generators} implies

\begin{thm}
Let  $\ell$ be odd or $\ell\equiv 2 \mod 4$. Then there is an
isomorphism of commutative, associative rings
\begin{eqnarray*}
\Psi: \mathcal F_q(\mathfrak{so}_{2n}, \ell) \simeq
\mZ[\chi({\omega_1}),\ldots,\chi({\omega_n})]/I_\ell
\end{eqnarray*}
which sends $[\la]$ to the character $\chi(\la)$. Here $I_\ell$ is
generated by all characters $\chi(\la)$ where $\la_1+\la_2=k+1$.
In case $\ell\equiv 0 \mod 4$ the ideal $I_\ell$ is generated by all
characters $\chi(\la)$ where $\la_1+\la_2=k+1$ together with the
characters of all $T(\la)$, where  $2(\la_2-\la_{n-1})=k+2$.
\end{thm}

\begin{proof}
Note that  $m_1+2m_2+\cdots+ 2m_{n-2}+m_{n-1}+m_{n}=\la_1+\la_2$.
If $\ell$ is odd or $\ell\equiv 2 \mod 4$, then $k$ and $k+2$ are odd
and there is no minimal $\la\in X/\mathcal{A}_\ell$ with $\la_1+\la_2=k+2$. In
case $\ell\equiv 0 \mod 4$, $k$ is even and the elements $\la$ with
$2(\la_2-\la_{n-1})=2m_2+\cdots+ 2m_{n-2}=k+2$ are minimal.
\end{proof}

\section{Even and odd fusion rings for Type B}
In this section we assume that $\mathfrak g = \mathfrak {so}_{2n+1}$. The simple roots are now $\alpha_i = \epsilon_i -
\epsilon_{i+1}, i = 1, \cdots , n-1$ together with $\alpha_n =
\epsilon_n$. The corresponding fundamental weights are $\omega_i =
\epsilon_1 + \epsilon_2 + \cdots + \epsilon_i$,
$i = 1, \cdots , n-1$ and
$\omega_{n}=\frac{1}{2} (\epsilon_1+\cdots\epsilon_{n-1}+\epsilon_n)$.
The highest short respectively long roots are
\begin{eqnarray*}
\alpha_0&=&\alpha_1 + \alpha_2 + \alpha_3 + \cdots + \alpha_n =
\epsilon_1 =\omega_1\\
\beta_0&=&\alpha_1 + 2 \alpha_2 + 2 \alpha_3 + \cdots + 2\alpha_n =
\epsilon_1 + \epsilon_2=\omega_2.
\end{eqnarray*}

Then $\mathcal{A}_\ell$ consists of all
dominant weights $\sum_{i=1}^n m_i\omega_i$ such that
\begin{eqnarray*}
2m_1+2m_2+2m_3+\cdots+2m_{n-1}+m_{n}&\leq&\ell-2n=:k\\
m_1+2m_2+2m_3+\cdots 2m_{n-2}+2m_{n-1}+m_{n}&\leq&\ell/2-2n+1=:k
\end{eqnarray*}
in case $\ell$ is odd respectively even.

\begin{thm}
Let $\ell\geq 2n$ if $\ell$ is odd and $\ell\geq 4n-4$ if $\ell$ is
even. Then there is an isomorphism of commutative, associative rings
\begin{eqnarray*}
\Psi: \mathcal F_q(\mathfrak{so}_{2n+1}, \ell) \simeq
\mZ[\chi({\omega_1}),\ldots,\chi({\omega_n})]/I_\ell
\end{eqnarray*}
which sends $[\la]$ to the character $\chi(\la)$. Here $I_\ell$ is
\begin{enumerate}
\item $I_\ell=\langle \chi(\la) \mid 2\la_1=k+1\rangle=\langle
    \{\chi({\frac{k-1}{2}\omega_1+\omega_i})\mid 1\leq i\leq n-1\}\cup \{\chi(\frac{k-1}{2}\omega_1+2\omega_n)\}\rangle
    $ in case $\ell$ is odd,
\item $I_\ell=\langle \chi(\la)\mid \la_1+\la_2=k+1\rangle$  in
    case $\ell\equiv 0 \mod 4$,
\item $I_\ell=\langle \chi(\la), [T(\mu)]\mid \la_1+\la_2=k+1,
    \mu_1+\mu_2=k+2\rangle$  in case $\ell\equiv 2 \mod 4$.
\end{enumerate}
\end{thm}

\begin{proof}
The first equality in each case is a direct consequence from
Proposition \ref{generators}. To see the second equality in case
$\ell$ is odd we abbreviate $m=\frac{k+1}{2}$ and
 $\chi^s:=\chi({\frac{k-1}{2}\omega_1+\omega_s})$ and set
 $a_s=h_{m-1+s}-h_{m-1-s}$ for $1\leq s\leq n$. Then
\begin{lemma} $a_s=(-1)^{s+1}(\chi^s-h_1\chi^{s-1}+h_2\chi^{s-2}-h_3\chi^{i-3}+\cdots
\mp h_s\chi^1).$
\end{lemma}
\begin{proof}
We use the following determinant formula from \cite[Ex. 24.46]{FH}
\begin{eqnarray*}
\chi(\la)&=&|h_{\la_{j}-i+j}-h_{\la_i-i-j}|_{1\leq i,j\leq n}
\end{eqnarray*}
Then $\chi^1$ is given by a lower diagonal matrix with diagonal
entries all $1$ except of the top entry which is $h_m-h_{m-2}=a_1$.
The matrix for $\chi^2$ is of the form
\begin{eqnarray*}
\begin{pmatrix}
h_m-h_{m-2}&1&0&0&0&\cdots\\
h_{m+1}-h_{m-3}&h_1&1&0&0&\cdots\\
h_{m+2}-h_{m-4}&h_2&h_1&0&0&\cdots\\
h_{m+3}-h_{m-5}&h_3&h_2&1&0&\cdots\\
&\vdots&
\end{pmatrix}
\end{eqnarray*}
Expanding along the second row gives $-\chi^2=a_2-h_1a_1$, hence
$a_2=-\chi^2+h_1a_1$. The result follows inductively by computing
$\chi^s$ by expanding along the $s$th row.
\end{proof}
Expanding the determinant formula for $\chi(\la)$ with $2\la_1=k+1$
shows that the $a_i$ generate $I_\ell$.
Since obviously $\langle \chi({\frac{k-1}{2}\omega_1+\omega_i})\mid
1\leq i\leq n\rangle\subset I_\ell$, the equality follows.
\end{proof}

\section{An exceptional example: $G_2$}
\label{G2A}
Let $\mathfrak{g}$ be of type $G_2$. In Bourbaki notation the positive roots are $\alpha_1$, $\alpha_2$, $\alpha_1+\alpha_2$, $2\alpha_1+\alpha_2=\alpha_0=\omega_1$, $3\alpha_1+\alpha_2$, $3\alpha_1+2\alpha_2=\beta_0=\omega_2$ with $\check{\alpha}_0=2\check{\alpha}_1+3\check{\alpha}_2$, and $\check{\beta}_0=\check{\alpha}_1+2\check{\alpha}_2$. Note that $\langle \rho, \check{\beta}_0\rangle=3$ and $\langle \la+\rho, \check{\alpha}_0\rangle=5$. The definition of the {\it level} $k$ depends on the parity of $\ell$ and whether $3$ divides $\ell$ or not. If $3$ does not divide $\ell$, set $k=\ell-6$ for $\ell$ odd and $k = \ell/2 - 6$ for $\ell$ even. In case $3|\ell$ set $\ell=3\ell'$ and define $k = l'-4$ if $\ell$ is odd and $k = \ell'/2-4$ if $\ell$ is even. Then the fundamental alcove $\mathcal{A}_\ell$ equals
\begin{eqnarray*}
\mathcal{A}_\ell&=
\begin{cases}
\{\la\in X^+\mid\langle \la, \check{\beta}_0\rangle\leq k\}=
\{(a,b)\mid a+2b\leq k\},&\text{if $3|\ell$}, \\
\{\la\in X^+\mid\langle \la, \check{\alpha}_0\rangle\leq k\}=
\{(a,b)\mid 2a+3b\leq k\}, &\text{otherwise,}
\end{cases}
\end{eqnarray*}
where we abbreviate $(a,b)=a\omega_1+b\omega_2$.

\begin{thm}
The $G_2$ fusion ring is isomorphic to $\mZ[\chi({\omega_1}),\chi({\omega_2})]/I_\ell$, where $I_\ell$ is generated by
\begin{eqnarray*}
\begin{cases}
\chi(0,\frac{k+1}{2}),\chi(2,\frac{k-1}{2}) ,\chi(4,\frac{k-3}{2}),
&\text{if $3|l$ and $k$ odd,}\\
\chi(0,\frac{k}{2}+1)+\chi(0,\frac{k}{2}), \chi(1,\frac{k}{2}), \chi(3,\frac{k}{2}-1), &\text{if $3|\ell$, $k$ even,}\\
\chi(\la),\chi(\mu)+\chi(\mu-\omega_1),&\text{if $3\not|\ell$, $k\equiv 2\op{mod} 3$,}\\
\chi((0,\frac{k+2}{3}),\chi(\la), \chi(\mu)+\chi(\mu-\omega_1),&\text{if $3\not|\ell$, $k\equiv 1\op{mod} 3$.}
\end{cases}
\end{eqnarray*}
Here $\la$ runs through $\Lambda:=\{(a,b)\mid 2a+3b=k+1\}$ and $\mu\in\Lambda':=\{(a,b)\mid a\not=0, 2a+3b=k+2\}$.
\end{thm}

\begin{proof}
Abbreviate $L_i=[L(\omega_i)]$ for $i=1,2$ and let $J$ denote the ideal generated by the proposed generators of $I_\ell$.
Assume $3|\ell$. Let first $k$ be odd. The minimal elements in the sense of Proposition \eqref{generators} are the tilting modules of highest weights $(a,b)$ where $a+2b=k+1$. All weights belong to the upper closure of $\mathcal{A}_\ell$ so that $T((a,b))=\Delta((a,b))=L((a,b))$. Hence it is enough to show that $g_j:=[\Delta(2j,\frac{k+1}{2}-j)]\in J$ for $0\leq j\leq k$. From the linkage principle we obtain the classes of tilting modules $t_i:=[T(2i+1,\frac{k+1}{2}-i)]=[\Delta(2i+1,\frac{k+1}{2}-i)]+[\Delta(2i+1,\frac{k+1}{2}-i-1)]$ for $0\leq i\leq k$ and $s_i=[T(2i,\frac{k+3}{2}-i)]=[\Delta(2i,\frac{k+3}{2}-i)]+[\Delta(2i,\frac{k+3}{2}-i-2)]$ for $0\leq j\leq k+1$. Set $t_i=g_i=s_i=0$ if $i<0$. Using \eqref{eq5} or alternatively the type $G_2$ Littlewood-Richardson rule, \cite{Littelmann}, we obtain
\begin{equation}
\label{Aodd}
g_r L_1=
\begin{cases}
t_0+g_1,&\text{if $r=0$},\\
g_{r-1}+t_{r-1}+g_r+t_r+g_{r+1},&\text{if $r>0$};
\end{cases}
\end{equation}
\begin{equation}
g_rL_2=
\label{Bodd}
\begin{cases}
s_0+g_0+g_1+t_1,&\text{if $r=0$},\\
t_{r-2}+g_{r-1}+t_{r-1}+s_r+2g_{r}+t_{r}+g_{r+1}+t_{r+1},&\text{if $r>0$};
\end{cases}\\
\end{equation}
\begin{equation}
\label{Codd}
t_rL_1=
t_{r-1}+s_{r}+2g_r+t_r+s_{r+1}+2g_{r+1}+t_{r+1},\quad \text{ for $r\geq 0$}.
\end{equation}
Using \eqref{Aodd} twice \eqref{Bodd}, \eqref{Codd}, \eqref{Bodd}, \eqref{Codd} we obtain $t_0, t_1, s_0, s_1, t_2, s_2 \in J$. Repeatedly applying   \eqref{Aodd}, \eqref{Bodd}, \eqref{Codd} we obtain that $g_a, t_b, s_b\in J$ for $a\geq 3$, $b\geq 1$ and the first case of the theorem follows.

For even $k$ we have an additional minimal generator in the sense of Proposition \eqref{generators}, namely the class of $T(\la)$ where $\la=(0,\frac{k+2}{2})$ which belongs to the second alcove $s_{\alpha_0,1}\mathcal{A}_\ell$. Hence $[T(\la)]=[\Delta(\la)]+[\Delta(\la-\omega_2)]$. Set  $g_0=s_0=0$ and otherwise $g_j:=[\Delta(2j-1,\frac{k}{2}-j+1)]$, $t_j:=[T(2i,\frac{k}{2}-j+1)]=[\Delta(2j,\frac{k}{2}-j+1)]+[\Delta(2j,\frac{k}{2}-j)]$ and $s_j=[T(2i-1,\frac{k}{2}-i+2)]=[\Delta(2i,\frac{k}{2}-i+1)]+[\Delta(2i,\frac{k}{2}-i-1)]$ for all  $0\leq j\leq \frac{k}{2}+1$. It is enough to show that $g_j\in J$ for $0\leq j\leq \frac{k}{2}+1$. Formula \eqref{eq5} or alternatively the Littlewood-Richardson rules give us
\begin{equation}
\label{A}
g_rL_1=g_{r-1}+t_{r-1}+g_r+t_r+g_{r+1},\quad \text{ for $r\geq 0$};
\end{equation}
\begin{equation}
g_rL_2=
\label{B}
\begin{cases}
s_1+2g_1+t_1+g_2+t_2,&\text{if $r=1$},\\
t_{r-2}+g_{r-1}+t_{r-1}+2g_r+s_r+t_r+g_{r+1}+t_{r+1},&\text{if $r>1$;}
\end{cases}\\
\end{equation}
\begin{equation}
t_rL_1=
\label{C}
\begin{cases}
s_1+2g_1+t_1,&\text{if $r=0$},\\
t_{r-1}+s_{r}+2g_r+t_r+s_{r+1}+2g_{r+1}+t_{r+1},&\text{if $r>0$}.
\end{cases}
\end{equation}
Using \eqref{A}, \eqref{C}, \eqref{B} repeatedly we obtain
\begin{equation}
t_1,s_1,t_2,g_3,s_2,t_3, \ldots g_a,s_{a-1},t_a\ldots \in J=\langle t_0,g_1,g_2\rangle.
\end{equation}
Hence all $g$'s are in $J$ and the second case of the theorem follows.

Assume now that $3$ does not divide $\ell$. The minimal elements in the sense of Proposition \eqref{generators} are the tilting modules of highest weight $(a,b)$ where $(a,b)\in\Lambda\cup\Lambda'$  and additionally $(a,b)=(0,\frac{k+2}{2})$ if $k=1 \mod 3$. The weights from $\Lambda$ belong to the upper closure of ${A}_\ell$ and hence we have $T(\la)=\Delta(\la)=L(\la)$ for $\la\in\Lambda$. The weights from $\Lambda'$ belong to the second alcove $s_{\alpha_0,1}(\mathcal{A}_\ell)$ and hence we have $[T(\mu)]=[\Delta(\mu)]+[\Delta(\mu-\omega_1)]$ for $\mu\in\Lambda'$ by the linkage principle. Since  $(0,\frac{k+2}{2})$ is minimal in $X^+$ with respect to the strong linkage principle we must have $T(0,\frac{k+2}{3})=\Delta(0,\frac{k+2}{3})=L(0,\frac{k+2}{3})$. Hence the theorem follows from Proposition \ref{generators}.
\end{proof}

\excise{\section{Example $D_4$}
$m_1+2m_2+m_3+m_4=k+1$. Take $k=1$.
The weights in the fundamental alcove in $\omega$ (LHS) and $\epsilon$
(RHS) basis
\begin{eqnarray*}
(2,0,0,0)&& (2,0,0,0)\\
(0,1,0,0)& &(1,1,0,0)\\
(0,0,2,0)& &(1,1,1,-1)\\
(0,0,0,2)& &(1,1,1,1)\\
(1,0,1,0)& &(\frac{3}{2},\frac{1}{2},\frac{1}{2},-\frac{1}{2})\\
(1,0,0,1)& &(\frac{3}{2},\frac{1}{2},\frac{1}{2},\frac{1}{2})\\
(0,0,1,1)& &(1,1,1,0)
\end{eqnarray*}
Shrawan's generators:
$\chi(\la)$ for
$$\la\in\{2\omega_1, 3\omega_1, 4\omega_1, 5\omega_1, 6\omega_1,
\omega_1+\omega_3, \omega_1+\omega_4\}$$
our generators: $T({\la})$ for
$$\la\in\{2\omega_1, \omega_2,2\omega_3,2\omega_4,\omega_1+\omega_3,
\omega_1+\omega_4, \omega_3+\omega_4\}$$
}

\end{document}